\documentclass[12pt]{amsart}
﻿
﻿
﻿
\makeatletter \@addtoreset{equation}{section}

\makeatother \makeatletter
\newtheorem{theorem}{Theorem}[section]
\newtheorem{corollary}[theorem]{Corollary}
\newtheorem{definition}[theorem]{Definition}
\newtheorem{lemma}[theorem]{Lemma}
\newtheorem{proposition}[theorem]{Proposition}

\newtheorem{example}[theorem]{Example}
\newtheorem*{theorem*}{Theorem}
\newtheorem*{proposition*}{Proposition}﻿
\newtheorem*{lemma*}{Lemma}

\def\pr1{\prod\hskip -2.07ex * \hskip 0.9 ex}

﻿
\usepackage{latexsym} \usepackage{amsmath} \usepackage{amssymb}
\usepackage{amsthm}
\usepackage{centernot}
\usepackage{cancel}\usepackage{color}
﻿
\begin{document}
\noindent
\title{On a definitition of slice regular algebraic varieties}
\author{Anna Gori $^1$} 
\thanks{$^1$ Dipartimento di Matematica - Universit\`a di Milano,                
	Via Saldini 50, 20133  Milano, Italy}
\author{Giulia Sarfatti $^2$}
\thanks{ $^2$ DIISM - Universit\`a Politecnica delle Marche,               
	Via Brecce Bianche 12,  60131, Ancona, Italy}
\author{Fabio Vlacci $^3$}
\thanks{$^3$ MIGe  - Universit\`a di Trieste, via Valerio 12/1, 34127 Trieste, Italy}
\begin{abstract}
	{	
In this paper  we prove a complete correspondence 
between irreducible slice regular algebraic sets and quasi prime radical right ideals in the ring of quaternionic slice regular polynomials. In particular we make use of the definition of quasi prime ideals to give a characterization of irreducibility of slice regular algebraic sets.
More precisely, we show that if $I$ is a  radical and quasi prime right ideal, then  $\mathbb S_{\mathcal V_c(I)}$, the symmetrization of $\mathcal V_c(I)$, is an irreducible slice regular algebraic set, where $\mathcal V_c(I)$ denotes the set of common zeros with commuting components of polynomials in $I$.
Combining this fact with the results  of our previous paper \cite{irred}, we obtain that for $I$ radical, $\mathcal V_c(I)$ is irreducible if and only if $I$ is quasi prime.  We also find a complete characterization of quasi prime radical right ideals, proving that either they are completely prime or the radical of the symmetrization $\mathcal S(J)$ of a completely prime right ideal $J$.
}
	\end{abstract}﻿

\maketitle

\section*{Acknowledgments}
The authors are partially supported by:
GNSAGA-INdAM via the project ``Hypercomplex function theory and
applications''; the first author is also partially supported by MUR
project PRIN 2022 ``Real and Complex Manifolds: Geometry and
Holomorphic Dynamics'', the second and the third authors are also
partially supported by MUR projects PRIN 2022 ``Interactions between
Geometric Structures and Function Theories''.

\section{Introduction}

A quaternionic version of the well-known Hilbert Nullstellensatz has recently been proved in its strong formulation in the setting of slice regular polynomials $\mathbb H[q_1,\ldots,q_n]$ (see \cite{ israeliani-3, Nul2}). This result motivated the introduction of the notion of quaternionic slice regular algebraic sets $\mathcal V(I)$, i.e. common zeros of polynomials in a right ideal $I$, and 
commuting slice regular algebraic sets
$\mathcal V_c(I)$, i.e. points in $\mathcal V(I)$ with commuting components. 
In the paper \cite{irred}, we started to investigate the relations between algebraic properties of a right ideal $I$ of slice regular polynomials with the geometric notion of irreducibility for the corresponding slice regular algebraic sets.
A slice regular algebraic set $\mathcal{V}(I)$ is said to be {\em irreducible} if it cannot be decomposed as the union of two proper slice regular algebraic subsets. 
Using the new notion of {\em quasi prime} right ideal (see Definition \ref{qpideal}), in one direction we were able to prove  
\begin{theorem*}[{\cite[Theorem 5.13]{irred}}]
Let $I$ be a radical right ideal of $\mathbb H[q_1,\ldots,q_n]$ such that $\mathcal V_c(I)$ is irreducible,
then I is a quasi prime right ideal.
\end{theorem*}
\noindent In the other direction, we had to make an additional assumption on the set $\mathbb S_{\mathcal V_c(I)}$ obtained by rotating simultaneously the coordinates of the points in $\mathcal V_c(I)$ (see Subsection \ref{symm} for a precise construction).  
\begin{theorem*}[{\cite[Theorem 5.19]{irred}}]
Let $I$ be a radical and {quasi prime} right ideal of the ring $\mathbb H[q_1,\ldots,q_n]$ such that the symmetrization $\mathbb S_{\mathcal V_c(I)}$ is irreducible as a commuting slice regular algebraic set. Then also $\mathcal V_c(I)$ is irreducible. 		
\end{theorem*}

\noindent The first aim of this paper is to improve this description, providing a complete characterization of irreducible commuting slice regular algebraic sets associated with radical ideals. Indeed we here prove that the assumption on the irreducibility of $\mathbb S_{\mathcal V_c(I)}$ in \cite[Theorem 5.19]{irred} is redundant (see Lemma \ref{SVirred}). 
\begin{lemma*}
Let $I$ be a quasi prime and radical right ideal in $\mathbb H[q_1,\ldots,q_n]$. Then $\mathbb S_{\mathcal V_c(I)}$ is irreducible. 
\end{lemma*}
\noindent We thus obtain the following characterization 
(see Theorem \ref{cpirred})
\begin{theorem*} 
Let $I$ be a radical right ideal of $\mathbb H [q_1,\ldots,q_n]$. Then the corresponding commuting slice regular algebraic set $\mathcal V_c(I)$ is irreducible if and only if $I$ is quasi-prime.
\end{theorem*}

\noindent In \cite{Cimpric}, Cimpri\v c introduces new notions of primality of ideals in the quaternionic setting. 
For our purposes we will use the characterization given in \cite[Lemma 2.3]{Cimpric} for {\em prime one-sided} ideals. To avoid misunderstanding we will refer to such ideals as {\em $C$-prime} ideals (see Definition \ref{Cprime}) 
and  in  Lemma \ref{Cprimelemma} we prove that  
a right ideal $I$  of $\mathbb H [q_1,\ldots,q_n]$ is radical and quasi prime  if and only if it is $C$-prime.\\
 Using the established correspondence between quasi prime radical ideals and irreducible commuting slice regular algebraic sets in Theorem \ref{char} we completely characterize
quasi prime radical ideals (or, equivalently, $C$-prime ideals).
\begin{theorem*}  A right radical ideal $I$ of $\mathbb H [q_1,\ldots,q_n]$  is quasi prime if and only if it is either completely prime or the radical of the symmetrization of a completely prime ideal. 
\end{theorem*}

\noindent This result for principal right  ideals in  $\mathbb H [q_1,\ldots,q_n]$  was already obtained in \cite{irred}, combining Theorem 6.4 with Propositions 6.1 and 6.3.

%
%

\section{Preliminaries}
\noindent Let $\mathbb{H}=\mathbb{R}+i\mathbb{R}+j\mathbb{R}+k\mathbb{R}$ be the four-dimensional algebra of the
quaternions and let $\mathbb{S}=\{J \in \mathbb{H} \ : \ J^2=-1\}$ denote
the two dimensional sphere of imaginary units in $\mathbb H$.  Then we can slice $\mathbb H$ as
\[ \mathbb{H}=\bigcup_{J\in \mathbb{S}}(\mathbb{R}+\mathbb{R} J),  
\]
where $\mathbb{C}_J:=\mathbb{R}+\mathbb{R} J$ is isomorphic to the complex plane $\mathbb{C}$, for any $J\in \mathbb{S}$. 
{Hence, any $q\in \mathbb{H}$ can be expressed as $q=x+yJ$ where $x,y \in \mathbb{R}$ and $J \in
	\mathbb{S}$. The {\it real part }of $q$ is ${\rm Re}(q)=x$ and its
	{\it imaginary part} is ${\rm Im}(q)=yJ$; the {\it conjugate}
	of $q$ is $\bar q:={\rm Re}(q)-{\rm Im}(q)$.  For any non-real
	quaternion $a\in \mathbb{H}\setminus \mathbb{R}$ we will
	denote by $J_a:=\frac{{\rm Im}(a)}{|{\rm Im}(a)|}\in
	\mathbb{S}$ and by $\mathbb{S}_a:=\{{\rm Re}(a)+J |{\rm
		Im}(a)| \ : \ J\in \mathbb{S}\}$. If $a\in \mathbb{R}$,
	then $J_a$ can be any imaginary unit.}
Similarly, if $\textbf{\em a}=(a_1,\ldots,a_n)\in \mathbb H^n$, then the
{\em spherical set} \[\mathbb S_{(a_1,\ldots,a_n)}=\mathbb S_{\textbf{\em a}}
:=\{(g^{-1}a_1g, \ldots, g^{-1}a_ng) : g \in \mathbb{H}\setminus\{0\} \}\]
consists of the points obtained by simultaneously rotating each coordinate of the point 
$\textbf{\em a}=(a_1,\ldots,a_n)$. 
If there exists $J\in \mathbb S$ such that $\textbf{\em a}=(a_1,\ldots,a_n) \in \mathbb C_J^n$, then $\mathbb S_{\textbf{\em a}}$ is called {\em arranged spherical set}.
﻿
Slice regular  polynomials in $n$ quaternionic variables are polynomial functions
$P:{\mathbb{H}^n}  \to \mathbb{H}$,
\[ P(q_1,\ldots, q_n)=
\sum_{ \substack{\ell_1=0,\ldots, L_1  \\
		\cdots\\ \ell_n=0,\ldots, L_n} }{q_1}^{\ell_1}\cdots {q_n}^{\ell_n}a_{\ell_1,\ldots,\ell_n} \]
with $a_{\ell_1,\ldots,\ell_n}\in\mathbb{H}$. ﻿

\noindent The set of slice regular polynomials can be endowed with an
appropriate notion of (non-commutative) product, the so called {\em slice product}, that
will be denoted by the symbol $*$.
Let us recall how it works for monomials, and extend it by linearity to polynomials. 
\begin{definition}
	If $M(q_1,\ldots,q_n)=q_1^{m_1}\cdots q_n^{m_n}a$ and $L(q_1,\ldots,q_n)=q_1^{l_1}\cdots q_n^{l_n}b$ belong to $\mathbb H[q_1,\ldots,q_n]$, then
	$$
	M*L (q_1,\ldots,q_n)=q_1^{m_1+l_1}\cdots q_n^{m_n+l_n}ab.
	$$
\end{definition}
\noindent In this way, $(\mathbb{H}[q_1,\ldots,q_n], +,*)$ is a Noetherian (see, e.g., \cite{reyes}) non-commutative ring.

\noindent For points with commuting components, the $*$-product {has the following explicit expression.}
\begin{proposition}[{\cite[Proposition 3.1]{Nul2}}]\label{prodstar} 
	Let ${  \textbf{a}}=(a_1,\ldots,a_n)\in \mathbb H^n$ be
	such that 
	$a_la_m=a_ma_l$ for any $1\leq l,m\leq n$
	and let $P,Q \in \mathbb H[q_1,\ldots,q_n]$. Then
	\[P*Q(\textbf{a})=\left\{\begin{array}{lr}
		0 & \text {if}\ P(\textbf{a})=0\\
		P(\textbf{a})\cdot Q(P(\textbf{a})^{-1}a_1P(\textbf{a}), P(\textbf{a})^{-1}a_2P(\textbf{a}),\ldots,P(\textbf{a})^{-1}a_nP(\textbf{a})) & \text{if}\ P(\textbf{a})\neq0
	\end{array}\right.\] 	
\end{proposition}

﻿

\noindent As in the one-variable case, it is possible to introduce two linear
operators on slice regular polynomials. For the sake of
simplicity, we define them only for monomials and extend it to polynomials by linearity.
\begin{definition}\label{R-coniugata2}
	Let $M(q_1,\ldots,q_n)=q_1^{m_1}\cdots q_n^{m_n}a\in \mathbb H[q_1,\ldots,q_n]$. Then the \emph{regular conjugate} of $M$ is 
	\[M^c(q_1,\ldots,q_n)=q_1^{m_1}\cdots q_n^{m_n}\overline{a},\]
	and the
	\emph{symmetrization} of $M$ is  
	\[M^s= M*M^c = M^c*M.\]
\end{definition}

\noindent In what follows, we will mainly consider {\em right ideals} of $\mathbb{H}[q_1,\ldots,q_n]$
  i.e. subgroups $I$ of $(\mathbb H[q_1,\ldots, q_n], +)$ such that for all $P\in I$  and all $Q\in\mathbb{H}[q_1,\ldots,q_n]$, $P*Q$ belongs to $I$.
		Analogously (reversing the order in the $*$-product) one can define {\em left} ideals of  
		$(\mathbb H[q_1,\ldots, q_n], +, *)$; subsets
		of  $(\mathbb H[q_1,\ldots, q_n], +, *)$ which are simultaneously right and left ideals will be called {\em two-sided} ideals.
		Furthermore, we will denote by 
	$\langle P : \ P \in \mathcal{E}\rangle $
	the {\em right} ideal generated in $(\mathbb H[q_1,\ldots, q_n], +, *)$ by the slice regular polynomials belonging to the subset $ \mathcal E \subset \mathbb H[q_1,\ldots, q_n]$.
In other words, $Q\in \mathbb{H}[q_1,\ldots,q_n]$ belongs to $\langle P : \ P \in \mathcal{E}\rangle $ 
	if and only if there exist a positive integer $N$ and slice regular polynomials $P_1,\ldots, P_N \in \mathcal E$ and $Q_1,\ldots, Q_N\in \mathbb{H}[q_1,\ldots,q_n]$  
	such that $P=\sum_{k=1}^N P_k*Q_k$. 
To any subset of $\mathbb H^n$ we can associate a right ideal as follows.
\begin{definition}\label{idealinuovi}
	Let $Z$ be any subset of $\mathbb{H}^n$.
	\noindent We denote by
	$\mathcal{J}(Z)$  the right ideal generated in $\mathbb{H}[q_1,\ldots,q_n]$
	by slice regular polynomials which vanish on $Z$: 
		a polynomial $P\in \mathbb{H}[q_1,\ldots,q_n]$ belongs to  $\mathcal{J}(Z)$  if there exist a positive integer $N$ and polynomials $P_1,\ldots, P_N, Q_1\ldots, Q_N\in \mathbb{H}[q_1,\ldots,q_n]$  such that $P_k$ vanishes on $Z$ for any $k=1,\ldots, N$ and $P=\sum_{k=1}^N P_k*Q_k$.
%
\end{definition}
\noindent In general 
$\mathcal{J}(Z)$ does not coincide with the set of polynomials vanishing on $Z$, but if $Z_c$ is any subset of $\mathbb H^n$ of points with commuting components, then it can be proved directly, using Proposition \ref{prodstar}, that
	\begin{equation}\label{ZcinVc}
		\mathcal J(Z_c)=\{P\in \mathbb H[q_1,\ldots,q_n] : \  P_{|_{Z_c}} = 0 \}.
\end{equation}
\noindent Furthermore, notice that if $Z'\subseteq Z$ then $\mathcal J(Z')\supseteq \mathcal J (Z)$.


\noindent For our purposes we need to consider the following two generalizations of prime ideals in the non-commutative setting.
The first one was introduced by Reyes in \cite{reyes}. 
\begin{definition}\label{cprime} A right ideal $I \subsetneq \mathbb{H}[q_1,\ldots,q_n]$  is said to be  {\em completely prime} if and only if given two polynomials $P,Q$ such that 
	$P*Q\in I$ 
	and  $P*I\subseteq I$, 
	then  either $P\in I$ or $Q\in I$.
\end{definition}
\noindent In \cite{irred} we introduced the following
\begin{definition}\label{qpideal} A right ideal $I\subsetneq  \mathbb H[q_1,\ldots,q_n]$ is said to be  {\em quasi prime} if and only if given two polynomials $P,Q$ such that 
	$P*Q\in I$, 
	then  either $P\in I$ or $Q^s\in I$.
\end{definition}
﻿\noindent The class of quasi prime ideals contains the one of completely prime ideals (see Proposition 5.11 in \cite{irred}).

\noindent Moreover, the notion of completely prime ideals is used to define the radical of an ideal.﻿
﻿
\begin{definition}
	Let $I$ be a right ideal of $\mathbb{H}[q_1,\ldots,q_n]$. Then the (right) radical of $I$ is defined as
	\[\sqrt I:= \bigcap_{\substack{ I \subset J\\ \; {J\text{ completely prime right ideal}}}}J.\]
	An ideal $I$  in $\mathbb{H}[q_1,\ldots,q_n]$  is said to be  {\em radical} if and only if $\sqrt I=I$.
\end{definition}
\noindent Notice that completely prime ideals are therefore radical and quasi prime. \\

\noindent In \cite{aryapoor}, Aryapoor gives an explicit characterization of the radical of an ideal in the non-commutative framework, which in our setting can be restated in this way.

\begin{proposition}[{\cite[Proposition 2.5]{aryapoor}}]\label{arya}
Let $I$ be a proper right ideal of $\mathbb H[q_1,\ldots,q_n]$. The (right) radical of $I$ coincides with the set of all polynomials $P\in \mathbb H[q_1,\ldots,q_n]$ such that for every $a\in \mathbb H$, there exists a natural number $N \ge 1$ satisfying the condition 
\[(Pa)^{*N}\in I + (Pa)*I+(Pa)^{*{2}}*I+\cdots+(Pa)^{*N}*I.\]
\end{proposition}
\noindent To any right ideal in $\mathbb H[q_1,\ldots, q_n]$ we assign two subsets of $\mathbb{H}^n$.
\begin{definition}
	\noindent Let $I$ be a right ideal in $\mathbb{H}[q_1,\ldots,q_n]$.
Then $\mathcal{V}(I)$ is the set of common zeros of $P\in I$, i.e.,
	\[ \mathcal{V}(I):=\bigcap_{P\in I} Z_P, \quad \text{ where $Z_P\subset \mathbb{H}^n$ denotes the zero set of $P\in I$ 
	}.\]
		\noindent Furthermore, we define $\mathcal{V}_c(I):=\mathcal{V}(I)\cap\bigcup_{K\in \mathbb{S}}\mathbb{C}_K^n.$
	\end{definition}

	\noindent ﻿The subsets of $\mathbb H^n$ associated with a right ideal and to its radical coincide (see Proposition 2.13 in \cite{irred}).   
	\begin{proposition}\label{Vrad}
		Let $I$ be a right ideal in $\mathbb H[q_1,\ldots,q_n]$. Then
		$\mathcal V(I)=\mathcal V(\sqrt{I})$ and $\mathcal V_c(I)=\mathcal V_c(\sqrt{I}).$		
	\end{proposition}

\noindent In analogy with the classical theory of algebraic geometry, it is possible to introduce a notion of algebraic set in $\mathbb H^n$, see \cite{Nul2}.
﻿

\begin{definition} \label{slicealg}
	A subset $V\subseteq \mathbb H^n$ is called {\em slice algebraic} if for any $K\in \mathbb S$, $V\cap \mathbb C_K^n$ is a complex algebraic subset of $\mathbb C_K^n$.
\end{definition}

\noindent As proven in \cite{Nul2}, all sets of the form $\mathcal V(I)$ and $\mathcal V_c(I)$ are slice algebraic, and due to their nature they are called {\em slice regular algebraic sets} and {\em commuting} slice regular algebraic sets.

\noindent Combining Theorem 4.1 and Corollary 4.2 in \cite{Nul2} we have 
\begin{proposition} \label{vanishing}
	Let $I$ be a right ideal of $\mathbb H [q_1,\ldots,q_n]$. Then
	\begin{enumerate}
		\item $\mathcal{J}(\mathcal{V}_c(I))=\mathcal{J}(\mathcal{V}(I))$; 
		\item $\mathcal{J}(\mathcal{V}(I))$ coincides with the set of polynomials vanishing on $\mathcal V(I)$.
	\end{enumerate}
\end{proposition}
\noindent Therefore it is not difficult to show that, if $I_1$ and $I_2$ are two right ideals in $\mathbb H[q_1,\ldots,q_n]$,
\begin{equation}\label{Jcap1}
	\mathcal{J}(\mathcal{V}(I_1)\cup \mathcal{V}(I_2))=\mathcal{J}(\mathcal{V}(I_1))\cap \mathcal{J}(\mathcal{V}(I_2)).
\end{equation}
%
\noindent Moreover, since $\mathcal J$ reverses the inclusions we also have
\begin{equation}\label{Jcup1}
	\mathcal{J}(\mathcal{V}(I_1)\cap \mathcal{V}(I_2))\supseteq\mathcal{J}(\mathcal{V}(I_1))\cup \mathcal{J}(\mathcal{V}(I_2)).
\end{equation}
\noindent We can now recall the Strong Nullstellensatz for slice regular polynomials (\cite[Theorem 4.3]{Nul2})
﻿\begin{theorem}[{Strong Nullstellensatz}]\label{snull}
	Let $I$ be a right ideal of $\mathbb{H}[q_1,\ldots,q_n]$. Then
	\[\mathcal{J}(\mathcal{V}(I))=\sqrt I.\]
\end{theorem}

\subsection{Spherical symmetry}\label{symm}

	Starting from a set $E_c\subseteq \mathbb H^n$ containing only points with commuting components, we denote its {\em symmetrization} by $$\mathbb{S}_{E_c}:=\bigcup_{\textbf{a} \in E_c}\mathbb S_{\textbf a}.$$

	﻿
	\noindent {Note that, if all the points of $E_c$ have commuting components, then the same holds true for all the points of $\mathbb{S}_{E_c}$.﻿}
	
	\begin{definition} 
		Given a right ideal $I\subseteq  \mathbb H[q_1,\ldots,q_n]$, we can define the {\em symmetrized ideal} $\mathcal{S}(I)$ as the ideal generated by $P^s$ with $ P\in I$, i.e. 
		$\mathcal{S}(I)=\langle P^s\;|\;P\in I\rangle$.
	\end{definition}
	\noindent Notice that $ \mathcal{S}(I)$ is contained in $I$. 
The symmetrization of a $\mathcal V_c(I)$ and the symmetrized ideal $\mathcal S(I)$ are related by the following results (see \cite{irred}))
	\begin{theorem}\label{simmalg}
		Let $I$ be a right ideal of $\mathbb H[q_1,\ldots,q_n]$.  Then 
		\begin{equation}\label{inclusion}
			\mathbb{S}_{\mathcal{V}_c(I)}=	\mathcal{V}_c( \mathcal{S}(I)).
		\end{equation}
	\end{theorem}
\noindent We thus have that $\mathbb S_{\mathcal V_c(I)}$ is a commuting slice regular algebraic set.
\noindent From the algebraic viewpoint we have the following description.
\begin{corollary}\label{essealg}
	Let  $I$ be an ideal  in $\mathbb H[q_1,\ldots,q_n]$. Then 
	\[\mathcal J (\mathbb S_{\mathcal V_c(I)})=\sqrt{\mathcal S(I)}.\]
Furthermore $\mathcal J (\mathbb S_{\mathcal V_c(I)})$ is a two-sided ideal.

	\end{corollary} 

\begin{definition}
		We say that a commuting slice regular algebraic set $\mathcal V_c(I)$ has {\em spherical symmetry} if $\mathcal V_c(I) = \mathcal V_c(\mathcal S(I)) = \mathbb S_{\mathcal V_c(I)}$. 
	\end{definition}
%

﻿\noindent As a particular case, we have the following (see Proposition 3.10 in \cite{irred})
\begin{proposition}\label{twosided}
	Let $I$ be a two-sided ideal of $\mathbb H[q_1,\ldots, q_n]$. Then $\mathcal V_c(I)$ has spherical symmetry.
\end{proposition}

\noindent Furthermore we have the following equivalence (see Proposition 3.12 in \cite{irred}).
\begin{proposition}\label{equivalence}
{The commuting slice regular algebraic set}	$\mathcal V_c(I)$ has spherical symmetry if and only if $\sqrt{I}$ is a two-sided ideal.
\end{proposition}
\section{Slice regular algebraic varieties}
	\noindent ﻿{Let us now recall the definition of reducibility of slice regular algebraic sets as introduced in \cite{irred}.}
﻿
\begin{definition} \label{red}
	Given a right ideal $I\subseteq \mathbb{H}[q_1,\ldots,q_n]$, the slice regular algebraic set $\mathcal{V}(I)$, 
	is said to be {\em reducible} whenever there exist two right ideals $I_1$ and $I_2$ in $\mathbb{H}[q_1,\ldots,q_n]$, such that $\mathcal{V}(I_1)\subsetneq \mathcal{V}(I)$, $\mathcal{V}(I_2)\subsetneq \mathcal{V}(I)$ and
	$$\mathcal{V}(I)=\mathcal{V}(I_1)\cup \mathcal{V}(I_2).$$
	Analogously, $\mathcal{V}_c(I)$ is said to be {\em reducible} if there exist two right ideals $I_1$ and $I_2$ in $\mathbb{H}[q_1,\ldots,q_n]$, such that $\mathcal{V}_c(I_1)\subsetneq \mathcal{V}_c(I)$, $\mathcal{V}_c(I_2)\subsetneq \mathcal{V}_c(I)$
	$$\mathcal{V}_c(I)=\mathcal{V}_c(I_1)\cup \mathcal{V}_c(I_2).$$
	﻿
	\noindent If $\mathcal{V}(I)$ (or $\mathcal{V}_c(I)$) is not reducible, it is called {\em irreducible}. 
\end{definition} %

{\begin{definition}
		An irreducible {(commuting)} slice regular algebraic set will be called a {\em (commuting) slice regular algebraic variety}.
	\end{definition}	
}\noindent As proven in \cite[Proposition 5.2]{irred}, if $\mathcal V_c(I)$ is irreducible, then $\mathcal V(I)$ is irreducible as well, while the contrary does not hold (see Remark 5.7 in \cite{irred}).	

\noindent We now prove the fundamental lemma needed to establish one of the main results of the paper. 
\begin{lemma}\label{SVirred} Let $I$ be a quasi prime radical right ideal in $\mathbb H[q_1,\ldots,q_n]$. Then $\mathcal{V}_c{(\sqrt{\mathcal S(I)})}=\mathbb{S}_{\mathcal{V}_c(I)}$
{is irreducible}.
\end{lemma} 
\begin{proof} Suppose, by contradiction, that $$\mathbb{S}_{\mathcal{V}_c(I)}=\mathcal{V}_c(I_1)\cup\mathcal{V}_c(I_2)$$ with $\mathcal{V}_c(I_1)$ and $\mathcal{V}_c(I_2)$ strictly contained in $\mathbb S_{\mathcal{V}_c(I)}$. Notice that we can assume without loss of generality that $I_\ell$ is radical for $\ell=1,2$. We thus have  $$\mathbb{S}_{\mathcal{V}_c(I)}=\mathbb{S}_{\mathcal{V}_c(I_1)}\cup\mathbb{S}_{\mathcal{V}_c(I_2)}.$$
Consider first the case in which $\mathcal V_c(I_\ell)=\mathbb S_{\mathcal V_c(I_\ell)}$ for $\ell=1,2$. 

\noindent 
In this case, it is possible to find a polynomial $P$ in $I_1\setminus{I}$. In fact, if $I_1\subseteq I$, then 
$\mathcal{V}_c(I)\subseteq \mathcal{V}_c(I_1)$ and 
$$\mathbb{S}_{\mathcal{V}_c(I)}=\mathbb{S}_{\mathcal{V}_c(I_1)}\cup\mathbb{S}_{\mathcal{V}_c(I_2)} \subseteq \mathbb{S}_{\mathcal{V}_c(I_1)}=\mathcal{V}_c(I_1)$$ 
which is in contradiction with our assumption.\\ 
Take now $P\in I_1\setminus{I}$ and, for every $Q\in I_2$, consider $P*Q$. For any ${\textbf {\em a}}\in \mathcal{V}_c(I)\subseteq \mathbb S_{\mathcal V_c(I)}$ we have that ${\textbf {\em a}}$ belongs to $\mathbb{S}_{\mathcal{V}_c(I_1)}$ or to $\mathbb{S}_{\mathcal{V}_c(I_2)}$. 

 \noindent If ${\textbf {\em a}}\in\mathbb{S}_{\mathcal{V}_c (I_1)}=\mathcal{V}_c(I_1)$, since $P\in I_1$, we have that $P({\textbf {\em a}})=0$. 
 
\noindent If ${\textbf {\em a}}\in\mathbb{S}_{\mathcal{V}_c(I_2)}\setminus \mathbb{S}_{\mathcal{V}_c (I_1)}\subseteq {\mathcal{V}_c(I_2)}$, we have that $Q(P^{-1}({\textbf {\em a}}){\textbf {\em a}}P({\textbf {\em a}}))=0$ that is $(P*Q)({\textbf {\em a}})=0,$ hence $P*Q\in \mathcal J (\mathcal V_c(I))= \sqrt{I}=I$. Now recall that $I$ is quasi prime. It then follows that either $P\in I$, which is not the case, or $Q^s\in I$. This holds for every $Q\in I_2$, thus $$\mathcal S(I_2)\subseteq I,\;\;\mbox{and}\;\;\mathcal{V}_c(I)\subseteq \mathcal{V}_c(\mathcal S(I_2))=\mathbb{S}_{\mathcal{V}_c(I_2)}.$$ We thus have 
$$\mathbb{S}_{\mathcal{V}_c(I)}=\mathbb{S}_{\mathcal{V}_c(I_1)}\cup\mathbb{S}_{\mathcal{V}_c(I_2)} \subseteq \mathbb{S}_{\mathcal{V}_c(I_2)}=\mathcal{V}_c(I_2)$$ 
again in contradiction with our assumption.\\
\noindent {Consider now the case in which  $\mathcal V_c(I_1)\neq \mathbb S_{\mathcal V_c(I_1)}$.
 This is possible if and only if we can decompose $\mathcal V_c(I_1)=\mathbb S_{E_1} \cup F_1$ where $\mathbb S_{E_1}$ is the maximal subset of $\mathcal V_c(I_1)$ obtained as the symmetrization of a (non-empty) set of points with commuting components and $F_1$ is a non-empty set of points with commuting components, such that for every $\mathbf b \in F_1$, $\mathbb S_{\mathbf b}\cap F_1 =\{\mathbf b\}$. Moreover $\mathbb S_{F_1}\subset {\mathcal V_c(I_2)}$, because $\mathbb S_{\mathcal V_c(I)}$ has spherical symmetry.  Denote as $V_1$ the set $\mathbb S_{E_1}$. 
%
\noindent {Similarly, one can write 
 $\mathcal V_c(I_2)$   as the union of a (non-empty) symmetrized set $V_2=\mathbb S_{E_2}$ and a (possibly empty) set  $F_2$  such that $\mathbb S_{F_2}\subset \mathbb S_{E_1}\subset{\mathcal V_c(I_1)}$ and,  for every $\mathbf c \in F_2$, $\mathbb S_{\mathbf c}\cap F_2 =\{\mathbf c\}$. 
}

\noindent Note that  $\mathbb{S}_{\mathcal{V}_c(I)}=V_1\cup V_2$. {
Furthermore, since both $V_\ell$, for $\ell=1,2$, consist of points with commuting components, 
recalling Formula \eqref{ZcinVc}, we have that
 $V_\ell\subseteq \mathcal{V}_c(\mathcal{J}(V_\ell))$, for $\ell=1,2$.
Moreover, 
since $V_\ell\subseteq  \mathcal V_c(I_\ell),$ then   $\mathcal{J}(V_\ell)\supseteq \mathcal{J} (\mathcal V_c(I_\ell))=\sqrt{I_\ell}=I_\ell$ and hence $\mathcal {V}_c(\mathcal{J}(V_\ell))\subseteq \mathcal V_c(I_\ell)\subsetneq \mathbb S_{\mathcal V_c(I)}.$
}
We  thus have
$$\mathbb{S}_{\mathcal{V}_c(I)}=V_1\cup V_2\subseteq \mathcal{V}_c(\mathcal{J}(V_1))\cup \mathcal{V}_c(\mathcal{J}(V_2))\subseteq  \mathcal V_c(I_1)\cup  \mathcal V_c(I_2)=\mathbb{S}_{\mathcal{V}_c(I)}.$$ Now, from Proposition 3.8 in \cite{irred}  we get that $\mathcal{J}(V_\ell)$ is a two-sided ideal  and  
 Proposition 3.10 in \cite{irred} implies that $\mathcal{V}_c(\mathcal{J}(V_\ell))$ has spherical symmetry. Thus $\mathbb{S}_{\mathcal{V}_c(I)}$  admits another decomposition  as the non-trivial union of two slice regular algebraic subsets $\mathcal{V}_c(\mathcal{J}(V_\ell))$ with spherical symmetry. This cannot occur, as we have seen in the first part of the proof.}
\end{proof}
\noindent Lemma \ref{SVirred} allows us to remove the hypothesis on the irreducibility of $\mathbb S_{\mathcal V_c(I)}$ in  \cite [Theorem 5.19]{irred}. Combining these results with \cite[Theorem 5.13]{irred} we obtain the following characterization.
\begin{theorem}\label{cpirred}
		Let $I$ be a radical right ideal of $\mathbb H[q_1,\ldots,q_n]$. Then $I$ is quasi prime if and only if $\mathcal V_c(I)$ is irreducible.
	\end{theorem}
	
\noindent A natural question is if any quasi prime right ideal is also radical. The following example shows that this is not the case.  
\begin{example}\label{ex1}
	Consider the ideal $I=\langle (q+1)^{*2}\rangle=\langle (q+1)^{2}\rangle$ in $\mathbb H[q]$ which is two-sided, since the generator has real coefficients, and quasi prime. In fact
	if $P*Q\in I$ then \[P*Q=(q+1)^2*K\] with $K\in \mathbb H[q]$, which implies that either $(q+1)^{2}$ is a factor of $P$, and hence $P\in I$, or $q+1$ is a factor of $Q$ and thus $(q+1)^2$ is a factor of $Q^s$, and hence $Q^s\in I$. \\
	\noindent However, $I$ is not radical. Indeed, thanks to Proposition \ref{arya}, one obtains that $q+1 \in \sqrt{I}$, since for any $a\in \mathbb H$, $((q+1)a)^{*2}=(q+1)^2a^{2}\in I$, but $q+1 \notin I$. 
	
\end{example}  	
{\noindent 
%
One can also investigate whether the quasi primeness of an ideal is inherited by its radical. For the moment we do not have counterexamples. However, as a consequence of Lemma \ref{SVirred} and of Theorem \ref{cpirred}, we can prove
\begin{corollary}
Let $I$ be a quasi prime radical right ideal of $\mathbb H[q_1,\ldots,q_n]$. Then $\sqrt{\mathcal S(I)}$ is quasi prime.
\end{corollary}
\begin{proof}
	The result follows from the fact that $\mathcal V_c(\sqrt{\mathcal S(I)})=\mathbb S_{\mathcal V_c(I)}$.
\end{proof}

\noindent Note that in general 
the fact that $\sqrt I$ is a quasi prime right ideal does not necessarily imply that $I$ is quasi prime. 
\begin{example}
Let $I=\langle(q+1)^3\rangle$. Then $\sqrt I=\langle q+1 \rangle$ since, as shown in Example \ref{ex1}, it contains the ideal generated by $q+1$ which is maximal. Thus $\sqrt I$ is also quasi prime. However $I$ is not quasi prime. In fact $(q+1)^2*(q+1)\in I$, but neither $(q+1)^2$ nor $(q+1)^s=(q+1)^2$ belong to $I$.
\end{example}
 }

\noindent Moreover, not all radical ideals are quasi prime. 
\begin{example}\cite[Example 5.16]{irred}
The ideal $I=\langle q_1^2+1,q_2^2+1\rangle$ in $\mathbb{H}[q_1,q_2]$, which is radical, is not quasi prime.
Indeed $q_1^2-q_2^2=(q_1-q_2)*(q_1+q_2)\in I$ but neither $(q_1-q_2)$ nor $(q_1+q_2)^s$ belong to $I$.
	\end{example}
\noindent Observe that the fact that in the previous example $I$ is not quasi prime could be also deduced by Theorem \ref{cpirred}, since 
\[\mathcal V_c(I)=(\mathbb S_i\times \mathbb S_i) \cap \bigcup_{J\in \mathbb S} \mathbb C_J^2=\mathbb S_{(i,i)} \cup \mathbb S_{(-i,i)}\]
is reducible. 
The same conclusion was obtained in \cite{irred}, by applying \cite[Theorem 5.13]{irred}. 

\section{A characterization of radical quasi prime ideals}

\noindent In \cite{Cimpric}, Cimpri\v c, at the beginning of Section 2, introduces new notions of primality of ideals in the ring of quaternionic polynomials with central variables. 
In particular he defines {\em left prime ideals} and characterizes them in \cite[Lemma 2.3]{Cimpric}. This definition reformulated in the ring of slice regular quaternionic polynomials (which is isomorphic to the ring considered in \cite{Cimpric}) reads as follows.
\begin{definition}\label{Cprime} A right ideal $I\subsetneq  \mathbb H[q_1,\ldots,q_n]$ is said to be  {\em $C$-prime} if and only if given two polynomials $P,R$ with $P\in \mathbb{H}[q_1,\ldots,q_n]$ and  $R\in \mathbb{R}[q_1,\ldots,q_n]$ such that 
	$P*R\in I$, 
	then  either $P\in I$ or $R\in I$.
\end{definition}
\noindent Thanks to \cite[Theorem 1.2]{Cimpric} it turns out that $C$-prime ideals are radical.\\
\noindent The $C$-primality of a right ideal $I$ is indeed equivalent to the request that $I$ is quasi prime and radical.
\begin{lemma}\label{Cprimelemma} A right ideal $I\subsetneq  \mathbb H[q_1,\ldots,q_n]$ is radical and quasi prime if and only if it is $C$-prime.
\end{lemma}
\begin{proof}
Let $I$ be a radical and quasi prime right ideal. If  $P*R\in I$ with   $R\in \mathbb{R}[q_1,\ldots,q_n]$  and  $P\in \mathbb{H}[q_1,\ldots,q_n],$ $P\notin I$ then $R^s=R^{*2}\in I$. Thus $R^{*2}$ vanishes on $\mathcal{V}_c(I)$ that is, recalling Proposition \ref{prodstar}, for all $\textbf{\em a}\in \mathcal{V}_c(I)$  $$0=R^{*2}(\textbf{\em a})={R(\textbf{\em a})}^2.$$ Thus  $R\in\mathcal{J}(\mathcal{V}_c(I))=I$ and hence $I$ is $C$-prime. \\ \noindent In the other direction,
let $I$ be $C$-prime.  If $P$ and $Q$  are such that $P*Q\in I$, then $P*Q*Q^c=P*Q^s\in I.$ Since $Q^s$ has real coefficients we get that either $P\in I$ or $Q^s\in I$, that is $I$ is quasi prime. Since $C$-prime ideals are radical we get the claim.\end{proof}
\noindent The notion of a $C$-prime right ideal introduced by Cimpri\v c is therefore suitable to characterize the irreducibility of commuting slice regular algebraic sets as in Theorem \ref{cpirred}.

%
\noindent In the rest of the section we will find a complete characterization of radical quasi prime right ideals in $\mathbb H[q_1,\ldots,q_n]$ (i.e of $C$-prime ideals). 

\begin{proposition} \label{notsf}
Let $I$ be a radical quasi prime  
	right ideal such that $\mathcal{V}_c(I)\neq \mathbb{S}_{\mathcal{V}_c(I)}$. Then $I$ is completely prime. 
\end{proposition} 
\noindent
{The same result has been obtained by Cimpri\v c in \cite[Lemma 4.5]{Cimpric}: the hypothesis $\mathcal{V}_c(I)\neq \mathbb{S}_{\mathcal{V}_c(I)}$ is indeed equivalent to the assumption that $I$ is not two-sided (see Proposition \ref{equivalence}). Here we give a different, more geometric, proof of the statement.
\begin{proof} If for any $P$ such that $P*I\subseteq I$ we have that $P \in I$, then $I$ is completely prime. \\
	Let now $P_0\in \mathbb H[q_1,\ldots,q_n]$ be such that $P_0*I\subseteq I$ and $P_0\notin I$. 
We want to show that for any $R\in\mathbb{H}[q_1\ldots,q_n]$ such that $P_0*R \in I$ we have that $R\in I$. 
Define \[ A_0=\{\,Q\in \mathbb{H}[q_1.\ldots,q_n]\mid P_0*Q\in I\,\}. \] 
Notice that $I$ is contained in $A_0$. Our aim is to obtain that $A_0=I$. 
First of all $A_0$ is a right ideal. Indeed, let $Q_1,Q_2\in A_0$. Then \[ P_0*(Q_1+Q_2)=P_0*Q_1+P_0*Q_2\in I, \] hence $Q_1+Q_2\in A_0$. Take now $L\in \mathbb{H}[q_1,\ldots,q_n]$, then for every $Q\in A_0$, \[ P_0*(Q*L)=(P_0*Q)*L\in I, \] because $I$ is a right ideal. Therefore $Q*L\in A_0$.    \\
Let us now prove that $\mathbb{S}_{\mathcal{V}_c(I)}=\mathbb{S}_{\mathcal{V}_c(A_0)}.$ \\
Since $I$ is quasi prime, for any $Q\in A_0$, we get $Q^s \in I$, that is $\mathcal{S}(A_0)\subseteq I$.  
Consequently, $ \mathcal{V}_c(I)\subseteq \mathcal{V}_c(\mathcal{S}(A_0)) =\mathbb{S}_{\mathcal{V}_c(A_0)}$, which implies $\mathbb{S}_{\mathcal{V}_c(I)}\subseteq \mathbb{S}_{\mathcal{V}_c(A_0)}.$ \\
On the other hand, $I\subseteq A_0$ implies that $\mathcal{S}(I)\subseteq \mathcal{S}(A_0)$, 
which yields $\mathbb{S}_{\mathcal{V}_c(I)}\supseteq \mathbb{S}_{\mathcal{V}_c(A_0)}.$
Combining the two inclusions we conclude that 
 \[ \mathbb{S}_{\mathcal{V}_c(I)}=\mathbb{S}_{\mathcal{V}_c(A_0)}. \] 
Suppose now that $I\subsetneq {A_0}.$ Observe that in this case \[ \mathcal{V}_c({A_0})\subsetneq  \mathcal{V}_c(I), \] since otherwise, applying the operator $\mathcal J$, we would obtain \[ A_0 \subseteq \sqrt{A_0}=\sqrt{I}=I\subsetneq A_0 , \] a contradiction. 
{Since $\mathcal{V}_c(I)\neq \mathbb{S}_{\mathcal{V}_c(I)}$, we can}
 decompose $\mathcal V_c(I)$ as in the proof of Lemma \ref{SVirred}, namely
$\mathcal V_c(I)=\mathbb S_{E} \cup F$ where $\mathbb S_{E}$ is the maximal subset of $\mathcal V_c(I)$ obtained as the symmetrization of a (non-empty) set of points with commuting components and $F$ is a non-empty set of points with commuting components, such that for every $\mathbf b \in F$, $\mathbb S_{\mathbf b}\cap F =\{\mathbf b\}$. 
Since $\mathbb S_{\mathcal V_c(I)}=\mathbb S_{\mathcal V_c(A_0)}$, we can decompose $\mathcal V_c(I)$ as
\[\mathcal{V}_c(I)=\mathcal{V}_c({A_0})\cup B \]
where $B$ is a non-empty subset such that $B=\mathbb S_B  \subseteq \mathbb S_E$.
%
%
%
 Now consider $\mathcal{J}(B)$ the right ideal associated with $B$. Since $B$ contains only points with commuting components we have that \[\mathcal J (B)=\{P\in \mathbb H[q_1,\ldots,q_n] \ : \  P \text{ vanishes on } B \}\supseteq I. \] Then \[ \mathcal{V}_c(\mathcal{J}(B))\subseteq \mathcal{V}_c(I). \] 
 Moreover, the above inclusion is indeed strict. In fact $\mathcal V_c(I)\neq \mathbb S_{\mathcal V_c(I)}$, while $\mathcal{V}_c(\mathcal{J}(B))$ has spherical symmetry since $\mathcal J(B)$ is a two-sided ideal.
 Thus we have
 \[ \mathcal{V}_c(I) =\mathcal{V}_c({A_0})\cup B \subseteq \mathcal{V}_c({A_0})\cup \mathcal{V}_c(\mathcal J(B))\subseteq \mathcal V_c(I),\] 
 and hence 
 \[\mathcal V_c(I)=\mathcal{V}_c({A_0})\cup \mathcal{V}_c(\mathcal J(B))\]  
 is a reducible slice algebraic set, in contrast with Theorem \ref{cpirred}.
 Therefore $A_0=I$ and $I$ is completely prime.
\end{proof}
Hence if $I$ is a radical quasi prime right ideal which is not completely prime, then necessarily $\mathcal V_c(I)=\mathbb S_{\mathcal V_c(I)}$ or, equivalently, $I=\sqrt{\mathcal S(I)}$.
We can indeed be more precise. 
\begin{proposition}\label{sisf} If $I\subsetneq \mathbb H[q_1,\ldots,q_n]$ is a quasi prime radical right ideal not completely prime, then there exists a completely prime right ideal $J$ such that $I=\sqrt{\mathcal{S}(J)}$.
\end{proposition}
\begin{proof}
Assume that $I$ is quasi prime but not completely prime. Then, from Proposition \ref{notsf}, $I=\sqrt{\mathcal S(I)}$.
\noindent Since $I$ is not completely prime, there exist $P_0,Q_0\in \mathbb H[q_1,\ldots,q_n]$ such that
\[
P_0* Q_0\in I, \ P_0* I\subseteq I, \quad
P_0\notin I,\quad
Q_0\notin I.
\]
\noindent Consider, as in the proof of Proposition \ref{notsf}, 
\[
A_0=\{Q \in \mathbb H[q_1,\ldots,q_n] \mid P_0* Q\in I\,\}.
\]
which, as above, is a right ideal containing $I$. In this case, the inclusion is strict since 
$Q_0\in A_0\setminus I$, because $P_0* Q_0\in I$ and $Q_0\notin I$.
\noindent Moreover, since $I$ is quasi prime, as in the proof of Proposition \ref{notsf}, $\mathbb S_{\mathcal V_c(A_0)}=\mathbb S_{\mathcal V_c(I)}$ and hence $\sqrt{\mathcal S(A_0)}=\sqrt{\mathcal S(I)}=I$.

\noindent It remains to determine whether $A_0$ is completely prime.\\
Let $Q_1,Q_2\in \mathbb H[q_1,\ldots,q_n]$ be  such that 
\[\ 
Q_1*Q_2\in A_0,
\ Q_1*A_0\subseteq A_0,\ \
\]
and assume that
$
Q_1\notin A_0$. We first prove that $A_0$ is quasi prime.
By the definition of $A_0$, we have
\[
P_0* Q_1* Q_2 \in I.
\]
If $Q_1 \notin A_0$, then
$P_0*Q_1 \notin I.$
Since $I$ is quasi prime, it follows that
$
Q_2^s \in I \subsetneq A_0.$
Hence, $A_0$ is quasi prime.
Moreover,
\[
A_0 \supsetneq I
= \sqrt{\mathcal S(I)}
= \sqrt{\mathcal S(A_0)}.
\]
Hence {
\[ 
 \mathcal{V}_c(A_0)\subsetneq  \mathbb{S}_{\mathcal{V}_c (A_0)}.\]}

\noindent Note that the {the same strict inclusion holds for the}  radical of $A_0$, $\sqrt{A_0}$, 
since
$
\mathcal V_c(A_0)=\mathcal V_c(\sqrt{A_0}).
$
Now two cases may occur.
Suppose first that $\sqrt{A_0}$ is quasi prime. Then Proposition \ref{notsf} applies and implies that
$\sqrt{A_0}$ is completely prime (being radical, quasi prime and such that 
$\mathbb{S}_{\mathcal{V}_c (\sqrt{A_0})}\neq \mathcal{V}_c(\sqrt {A_0}))$.
\noindent
If 
$\sqrt{A_0}$ is not quasi prime, from Theorem \ref{cpirred}, it turns out that
$
\mathcal V_c(\sqrt{A_0})=\mathcal V_c(A_0)
$
is reducible. Hence there exist  radical right ideals $I_1,I_2$ 
such that
\[
\mathcal V_c(A_0)=\mathcal V_c(I_1)\cup\mathcal V_c(I_2).
\]
{with
$
\mathcal V_c(I_1)\subsetneq \mathcal V_c(A_0),
$ and 
$\mathcal V_c(I_2)\subsetneq\mathcal V_c(A_0).
$ }

\noindent Passing to the symmetrization
\[
\mathbb{S}_{\mathcal V_c(A_0)}
 =\mathcal V_c(\mathcal {S}(A_0))
 =\mathcal V_c(\mathcal{S}(I))
 =\mathcal V_c(I),
\]
whereas
$
\mathbb{S}_{\mathcal V_c(A_0)}
=
\mathbb{S}_{\mathcal V_c(I_1)}
\cup
\mathbb{S}_{\mathcal V_c(I_2)}.
$
On the other hand, $\mathcal V_c(I)=\mathbb{S}_{\mathcal V_c(I)}$ is irreducible,  again by 
Theorem \ref{cpirred} since $I$ is radical and quasi prime.
Hence
\[
\mathbb{S}_{\mathcal V_c(I_1)}=\mathbb{S}_{\mathcal V_c(I)}
\quad\text{or}\quad
\mathbb{S}_{\mathcal V_c(I_2)}=\mathbb{S}_{\mathcal V_c(I)}.
\]
Assume for instance that
\[
\mathbb{S}_{\mathcal V_c(I_1)}=\mathbb{S}_{\mathcal V_c(I)}.
\]
Equivalently,
\[
\sqrt{\mathcal{S}(I_1)}=\sqrt{\mathcal{S}(I)}=I.
\]
If $\mathcal V_c(I_1)$ was irreducible, then, since $I_1$ is radical,
quasi prime and such that $\mathbb{S}_{\mathcal{V}_c({I_1})}\neq \mathcal{V}_c({I_1})$, it would follow that $I_1$ is completely prime.
Otherwise, $\mathcal V_c(I_1)$ is reducible. Then there exist proper radical ideals
$I_1^{1},I_1^{2}$ such that
\[
\mathcal V_c(I_1)
=
\mathcal V_c(I_1^{1})
\cup
\mathcal V_c(I_1^{2}).
\]\\
The same argument may be iterated. If one of the commuting slice regular varieties $V_c(I_j)$
 is irreducible, the corresponding
ideal is completely prime. Otherwise we continue decomposing it.
By Zorn's Lemma, this process yields to a maximal element
$
\widehat I
$
in the chain
\[
I
\subsetneq
\sqrt{A_0}
\subsetneq
I_1
\subsetneq
I_1^{1}
\subsetneq
\cdots
\subsetneq
\widehat I.
\]
The ideal $\widehat I$ is radical, quasi prime and such that $\mathbb{S}_{\mathcal{V}_c(\widehat{I})}\neq \mathcal{V}_c(\widehat{I})$, hence it is completely prime.
Moreover,
\[
\mathbb{S}_{\mathcal V_c(\widehat I)}
=\mathbb S_{\mathcal V_c(I)}, \ \text{ and hence } \
\sqrt{\mathcal{S}(\widehat I)}
=\sqrt{\mathcal{S}(I)}=
I.
\]
Finally, define
\[
J:=\widehat I.
\]

\end{proof}
\noindent Furthermore we obtain the following characterization.

\begin{theorem} \label{char} In $\mathbb{H}[q_1,\ldots, q_n]$ a  right ideal $I$ is  quasi prime and radical if and only if it is either completely prime or the radical of the symmetrization of a completely prime ideal. 
\end{theorem}
 \begin{proof} 	
 One direction of the equivalence easily follows from Propositions \ref{notsf} and \ref{sisf}.
 
\noindent For the other implication, if $I$ is completely prime, then $I$ is radical and Proposition 5.12 in \cite{irred} guarantees that $I$ is quasi prime.   	
Assume now that \[ I=\sqrt{\mathcal{S}(J)}, \] where $J$ is a completely prime right ideal. Since every completely prime ideal is quasi prime and radical, we have $J=\sqrt{J}. $ Moreover, $\mathcal{V}_c(J)$ is irreducible, hence, by Lemma \ref{SVirred}, $ \mathbb{S}_{\mathcal{V}_c(J)}$ is irreducible as well. Therefore, by Theorem \ref{cpirred}, $\sqrt{\mathcal{S}(J)}$ is quasi prime. 
\end{proof}

\end{document}